\documentclass{amsart}

\usepackage{amsfonts}
\usepackage{amssymb}
\usepackage{amsthm}
\usepackage{eucal}
\usepackage{tikz}

\theoremstyle{plain}
\newtheorem{theorem}{Theorem}[section]
\newtheorem{proposition}[theorem]{Proposition}
\newtheorem{lemma}[theorem]{Lemma}
\newtheorem{corollary}[theorem]{Corollary}

\theoremstyle{definition}

\theoremstyle{remark}

\newtheorem*{acknowledgment}{Acknowledgment}

\newcommand{\p}{\mathcal{P}}
\newcommand{\gen}[1]{\langle #1\rangle}
\newcommand{\rank}{\mathrm{rank}}
\newcommand{\frank}{\mathrm{frank}}
\newcommand{\normal}{\trianglelefteq}
\newcommand{\Aut}{\mathrm{Aut}}
\newcommand{\Z}{\mathbb{Z}}
\newcommand{\E}{\mathcal{E}}

\begin{document}
\title{Embeddings of (proper) power graphs of finite groups}
\author{A. Doostabadi}
\author{M. Farrokhi D. G.}
\address{Department of Pure Mathematics, Ferdowsi University of Mashhad, Mashhad, Iran}
\email{a.doostabadi@yahoo.com}
\address{Department of Pure Mathematics, Ferdowsi University of Mashhad, Mashhad, Iran}
\email{m.farrokhi.d.g@gmail.com}
\date{}

\begin{abstract}
The (proper) power graph of a group is a graph whose vertex set is the set of all (nontrivial) elements of the group and two distinct vertices are adjacent if one is a power of the other. Various kinds of planarity of (proper) power graphs of groups are discussed.
\end{abstract}
\maketitle
\section{Introduction}
Recently, there have been an increasing interest in associating graphs to algebraic structures and studying how the properties of the associated graphs influence the structure of the given algebraic structures. If $G$ is a group (or a semigroup), then the \textit{power graph} of $G$, denoted by $\p(G)$, is a graph whose vertex set is $G$ in which two distinct vertices $x$ and $y$ are adjacent if one is a power of the other, in other words, $x$ and $y$ are adjacent if $x\in\gen{y}$ or $y\in\gen{x}$.

Chakrabarty, Ghosh and Sen \cite{ic-sg-mks} investigated the power graph of semigroups and characterized all semigroups with connected or complete power graphs. In the case of groups, Cameron and Ghosh \cite{pjc-sg} showed that two finite abelian groups are isomorphic if and only if they have isomorphic power graphs. As a generalization, Cameron \cite{pjc} proves that finite groups with isomorphic power graphs, have the same number of elements of each order. Further properties of power graphs including planarity, perfectness, chromatic number and clique number are discussed by Doostabadi, Erfanian and Jafarzadeh in \cite{ad-ae-aj}.

Since the identity element in a group $G$ is adjacent to all other vertices in the power graph $\p(G)$, we may always remove the identity element and study the resulting graph called the \textit{proper power graph} of $G$. The proper power graph of $G$ is denoted by $\p^*(G)$. For further results concerning power graphs and proper power graphs, we may refer the interested reader to \cite{ad-ae-mfdg} and \cite{ad-mfdg}.

The aim of this paper is to study various kinds of planarity of (proper) power graphs. Indeed, we shall classify all groups whose (proper) power graphs are planar, outerplanar, ring graph, $1$-planar, almost planar, maximal planar, toroidal and projective. In what follows, $\omega(G)$ stands for the set of orders of all elements of a given group $G$, i.e., $\omega(G)=\{|x|:x\in G\}$. Also, a Frobenius group with kernel $K$ and a complement $H$ is denoted by $K\rtimes_F H$. The \textit{dot product} of two vertex transitive graphs $\Gamma_1$ and $\Gamma_2$, is the graph obtained from the identification of a vertex of $\Gamma_1$ with a vertex of $\Gamma_2$ and it is denoted by $\Gamma_1\cdot\Gamma_2$.
\section{Planarity of (proper) power graphs}
We begin with the usual notion of planarity. A graph $\Gamma$ is called \textit{planar} if there is an embedding of $\Gamma$ in the plane in which the edges intersect only in the terminals. A famous theorem of Kuratowski states that a graph is planar if and only if it has no subgraphs as a subdivision of the graphs $K_5$ or $K_{3,3}$ (see \cite{kk}). The following results give a characterization of all planar (proper) power graphs and have central roles in the proofs of our subsequent results. We note that a graph is said to be $\Gamma$-free if it has no induced subgraphs isomorphic to $\Gamma$.
\begin{theorem}[\cite{ad-ae-aj}]\label{planar}
Let $G$ be a group. Then $\p(G)$ is planar if and only if $\omega(G)\subseteq\{1,2,3,4\}$.
\end{theorem}
\begin{theorem}\label{properplanar}
Let $G$ be a group. Then the following conditions are equivalent:
\begin{itemize}
\item[(1)]$\p^*(G)$ is planar,
\item[(2)]$\p^*(G)$ is $K_5$-free,
\item[(3)]$\p^*(G)$ is $K_6$-free,
\item[(4)]$\p^*(G)$ is $K_{3,3}$-free,
\item[(5)]$\omega(G)\subseteq\{1,2,3,4,5,6\}$.
\end{itemize}
\end{theorem}
\begin{proof}
We just prove the equivalence of (1) and (5). The other equivalences can be establish similarly.

First assume that $\p^*(G)$ is a planar graph. If $G$ has an element of infinite order, then clearly $\p^*(G)$ has a subgraph isomorphic to $K_5$, which is a contradiction. Thus $G$ is a torsion group. Now, let $x\in G$ be an arbitrary element. If $|x|=p^m$ is a prime power, then $p^m-1\leq4$ and hence $|x|\leq5$ for $\gen{x}\setminus\{1\}$ induces a complete subgraph of $\p^*(G)$. Also, if $|x|$ is not prime power and $p^mq^n$ divides $|x|$, then the elements of $\gen{x}\setminus\{1\}$ whose orders divide $p^m$ together with elements whose orders equal $p^mq^n$ induce a complete subgraph of $\p^*(G)$ of size $p^m-1+\varphi(p^mq^n)$, where $\varphi$ is the Euler totient function. Since $\p^*(G)$ is planar, this is possible only if $|x|\leq6$, as required.

Conversely, assume that $G$ is a torsion group with $\omega(G)\subseteq\{1,2,3,4,5,6\}$. A simple verification shows that $\p^*(G)$ is a union (not necessary disjoint) of some $K_1$, $K_2$, $K_4$, friendship graphs and families of complete graphs on $4$ vertices sharing an edge in such a way that any two such graphs have at most one edge in common and any three such graphs have no vertex in common. Hence, the resulting graph is planar, which completes the proof.
\end{proof}

An \textit{$n$-coloring} of a graph $\Gamma$ is an assignment of $n$ different colors to the vertices of $\Gamma$ such that adjacent vertices have different colors. The \textit{chromatic number} $\chi(\Gamma)$ is the minimal number $n$ such that $\Gamma$ has an $n$-coloring. An \textit{$n$-star coloring} of $\Gamma$ is an $n$-coloring of $\Gamma$ such that no path on four vertices in $\Gamma$ is $2$-colored. The \textit{star chromatic number} $\chi_s(\Gamma)$ is the minimal number $n$ such that $\Gamma$ has an $n$-star coloring. Utilizing the above theorems we have:
\begin{corollary}
If $G$ is a group with planar (proper) power graph, then $\chi(\p^*(G))=\chi_s(\p^*(G))$.
\end{corollary}

A \textit{chord} in a graph $\Gamma$ is an edge joining two nonadjacent vertices in a cycle of $\Gamma$ and a cycle with no chord is called a \textit{primitive cycle}. A graph $\Gamma$ in which any two primitive cycles intersect in at most one edge is said to admit the \textit{primitive cycle property} (PCP). The \textit{free rank} of $\Gamma$, denoted by $\frank(\Gamma)$, is the number of primitive cycles of $\Gamma$. Also, the \textit{cycle rank} of $\Gamma$, denoted by  $\rank(\Gamma)$, is the number $e-v+c$, where $v,e,c$ are the number of vertices, the number of edges and the number of connected components of $\Gamma$, respectively. Clearly, the cycle rank of $\Gamma$ is the same as the dimension of the cycle space of $\Gamma$. By \cite[Proposition 2.2]{ig-er-rhv}, we have $\rank(\Gamma)\leqslant \frank(\Gamma)$. A graph $\Gamma$ is called a \textit{ring graph} if one of the following equivalent conditions holds (see \cite{ig-er-rhv}). 
\begin{itemize}
\item $\rank(\Gamma)=\frank(\Gamma)$,
\item $\Gamma$ satisfies the PCP and $\Gamma$ does not contain a subdivision of $K_4$ as a subgraph.
\end{itemize}

Also, a graph is \textit{outerplanar} if it has a planar embedding all its vertices lie on a simple closed curve, say a circle. A well-known result states that a graph is outerplanar if and only if it does not contain a subdivision of $K_4$ and $K_{2,3}$ as a subgraph (see \cite{gc-fh}). Clearly, every outerplanar graph is a ring graph and every ring graph is a planar graph.
\begin{theorem}
Let $G$ be a group. Then $\p(G)$ (resp. $\p^*(G)$) is ring graph if and only if $\omega(G)\subseteq\{1,2,3\}$ (resp. $\omega(G)\subseteq\{1,2,3,4\}$).
\end{theorem}
\begin{proof}
If $\p^*(G)$ is a ring graph, then by Theorem \ref{planar}, $\omega(G)\subseteq\{1,2,3,4,5,6\}$. If $G$ has an element of order $5$ or $6$, then $\gen{x}$ contains a subgraph isomorphic to $K_4$, which is impossible. Thus $\omega(G)\subseteq\{1,2,3,4\}$. Clearly, $w(G)\subseteq\{1,2,3\}$ when $\p(G)$ is a ring graph. The converse is obvious.
\end{proof}
\begin{corollary}
Let $G$ be a group. Then $\p(G)$ (resp. $\p^*(G)$) is outerplanar if and only if it is a ring graph.
\end{corollary}

A graph is called \textit{$1$-planar} if it can be drawn in the plane such that its edges each of which is crossed by at most one other edge.
\begin{theorem}[Fabrici and Madaras \cite{if-tm}]\label{1-planar}
If $\Gamma$ is a $1$-planar graph on $v$ vertices and $e$ edges, then $e\leq4v-8$.
\end{theorem}
\begin{corollary}
The complete graph $K_7$ is not $1$-planar.
\end{corollary}
\begin{proof}
Suppose on the contrary that $K_7$ is $1$-planar. Then, by Theorem \ref{1-planar}, we should have $21=e\leq 4v-8=20$, which is a contradiction.
\end{proof}

To deal with the case of $1$-planar power graphs, we need to decide on the $1$-planarity of a particular graph, which is provided by the following lemma.
\begin{lemma}\label{K9-K6+3K2}
Let $\Gamma$ be the graph obtained from $K_9\setminus K_6$ by adding three new disjoint edges. Then $\Gamma$ is not $1$-planar.
\end{lemma}
\begin{proof}
Le $u_1,u_2,u_3$ be the vertices adjacent to all other vertices, and $\{v_1,v_2\}$, $\{v_3,v_4\}$ and $\{v_5,v_6\}$ be the three disjoint edges whose end vertices are different from $u_1,u_2,u_3$. Suppose on the contrary that $\Gamma$ is $1$-planar and consider a $1$-planar embedding $\E$ of $\Gamma$ with minimum number of crosses.  If $\E$ has an edge crossing itself or two crossing incident edges, then one can easily unknot the cross and reach to a $1$-planar embedding of $\Gamma$ with smaller number of crosses contradicting the choice of $\E$. Hence, $\E$ has neither an edge crossing itself nor two crossing incident edges. This implies that the subgraph $\Delta$ induced by $\{u_1,u_2,u_3\}$ is simply a triangle. Using a direct computation one can show, step-by-step, that
\begin{itemize}
\item the edges incident to each of the vertices $v_1,\ldots,v_6$ crosses no more than one edges of $\Delta$,
\item at most one edge of $\Delta$ is crossed by an edge of $\Gamma$ different from $\{v_1,v_2\}$, $\{v_3,v_4\}$ and $\{v_5,v_6\}$,
\item the only edges of $\Gamma$ that can cross $\Delta$ are $\{v_1,v_2\}$, $\{v_3,v_4\}$ and $\{v_5,v_6\}$.
\end{itemize}
According to the above observations, we reach to the following $1$-planar drawing of $\Gamma$ in the interior region of $\Delta$ with maximum number of vertices (see Figure 1).

Clearly, there must exists a vertex $v_i$ outside $\Delta$ adjacent to some vertex $v_j$ inside $\Delta$ where $1\leq i,j\leq 6$. Hence, by Figure 1, we must have an edge in the interior region of $\Delta$ crossed more than once, leading to a contradiction.
\end{proof}

\begin{center} 

\begin{tikzpicture}
\node [circle,fill=black,inner sep=1pt] (A) at ({0.7*cos(90)},{0.7*sin(90)}) {};
\node [circle,fill=black,inner sep=1pt] (B) at ({0.7*cos(210)},{0.7*sin(210)}) {};
\node [circle,fill=black,inner sep=1pt] (C) at ({0.7*cos(330)},{0.7*sin(330)}) {};
\node [circle,fill=black,inner sep=1pt,label=above:{\tiny{$u_1$}}] (D) at ({2*cos(90)},{2*sin(90)}) {};
\node [circle,fill=black,inner sep=1pt,label=below:{\tiny{$u_2$}}] (E) at ({2*cos(210)},{2*sin(210)}) {};
\node [circle,fill=black,inner sep=1pt,label=below:{\tiny{$u_3$}}] (F) at ({2*cos(330)},{2*sin(330)}) {};

\draw (A)--(D)--(E)--(F)--(D)--(B)	--(E)--(A)--(F)--(C)--(D);
\draw (B)--(F);
\draw (C)--(E);
\end{tikzpicture}\\
Figure 1
\end{center}

Utilizing the same method as in the proof of Lemma \ref{K9-K6+3K2}, we obtain a new minimal non-$1$-planar graph, which is of independent interest.
\begin{proposition}\label{K9-K6+2K2}
Let $\Gamma$ be the graph obtained from $K_9\setminus K_6$ by adding two new disjoint edges. Then $\Gamma$ is a minimal non-$1$-planar graph.
\end{proposition}
\begin{theorem}
Let $G$ be a group. Then $\p(G)$ is $1$-planar if and only if $\omega(G)\subseteq\{1,2,3,4,5,6\}$ and any two cyclic subgroups of $G$ of order $6$ have at most two elements in common.
\end{theorem}
\begin{proof}
The same as in the proof of Theorem \ref{planar}, we can show that $\omega(G)\subseteq\{1,2,3,4,5,6\}$. If $\omega(G)\subseteq\{1,2,4,5\}$, then we are done. Thus we may assume that $3\in\omega(G)$. Let $g\in G$ be an element of order $3$. If $C_G(g)$ has two distinct involutions $x,y$, then it must have one more involution, say $z$, for $\gen{x,y}$ is a dihedral group. But then, by Lemma \ref{K9-K6+3K2}, the subgraph induced by elements of orders $1,3,6$ in $\gen{xg}\cup\gen{yg}\cup\gen{zg}$ is not $1$-planar giving us a contradiction. Thus no two distinct cyclic subgroups of $G$ of order $6$ have three elements in common. 

Conversely, if all the conditions are satisfied, then $\p(G)$ is a combination of induced subgraphs as drawn in Figure 2 in such a way that any these subgraphs have pairwise disjoint edges except possibly for a common edge whose end vertices are the trivial element and an involution. Therefore $\p(G)$ is $1$-planar, as required. Note that in Figure 2, $a,b,c,d,e$ denote elements of orders $2,3,4,5,6$, respectively.
\end{proof}

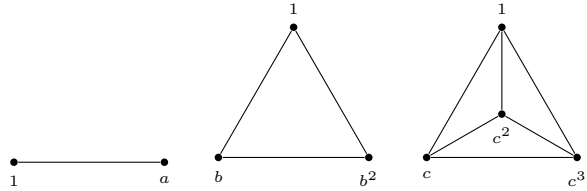
\begin{center} 
\begin{tabular}{ccc}
\begin{tikzpicture}
\node [circle,fill=black,inner sep=1pt,label=below:{{\tiny $1$}}] (A) at (0,0) {};
\node [circle,fill=black,inner sep=1pt,label=below:{{\tiny $a$}}] (B) at (2,0) {};
\draw (A)--(B);
\end{tikzpicture}
&
\begin{tikzpicture}
\node [circle,fill=black,inner sep=1pt,label=below:{{\tiny $b$}}] (A) at (0,0) {};
\node [circle,fill=black,inner sep=1pt,label=below:{{\tiny $b^2$}}] (B) at (2,0) {};
\node [circle,fill=black,inner sep=1pt,label=above:{{\tiny $1$}}] (C) at (1,{2*sin(60)}) {};
\draw (A)--(B)--(C)--(A);
\end{tikzpicture}
&
\begin{tikzpicture}
\node [circle,fill=black,inner sep=1pt,label=below:{{\tiny $c$}}] (A) at (0,0) {};
\node [circle,fill=black,inner sep=1pt,label=below:{{\tiny $c^3$}}] (B) at (2,0) {};
\node [circle,fill=black,inner sep=1pt,label=above:{{\tiny $1$}}] (C) at (1,{2*sin(60)}) {};
\node [circle,fill=black,inner sep=1pt,label=below:{{\tiny $c^2$}}] (D) at (1,{0.666*sin(60)}) {};
\draw (B)--(C)--(A)--(B)--(D)--(C);
\draw (A)--(D);
\end{tikzpicture}
\end{tabular}
\end{center}

\begin{center} 
\begin{tabular}{cc}
\begin{tikzpicture}
\node [circle,fill=black,inner sep=1pt,label=above:{{\tiny $1$}}] (A) at ({2*cos(90)},{2*sin(90)}) {};
\node [circle,fill=black,inner sep=1pt,label=below:{{\tiny $d^3$}}] (B) at ({2*cos(210)},{2*sin(210)}) {};
\node [circle,fill=black,inner sep=1pt,label=below:{{\tiny $d^4$}}] (C) at ({2*cos(330)},{2*sin(330)}) {};
\node [circle,fill=black,inner sep=1pt,label=10:{{\tiny $d^2$}}] (D) at ({cos(210)},{sin(210)}) {};
\node [circle,fill=black,inner sep=1pt,label=170:{{\tiny $d$}}] (E) at ({cos(330)},{sin(330)}) {};
\draw (A)--(B)--(C)--(D)--(E)--(A)--(C)--(E)--(B)--(D)--(A);
\end{tikzpicture}
&
\begin{tikzpicture}
\node [circle,fill=black,inner sep=1pt,label=left:{{\tiny $e^2$}}] (A) at ({cos(90)},{sin(90)}) {};
\node [circle,fill=black,inner sep=1pt,label=10:{{\tiny $e^4$}}] (B) at ({cos(210)},{sin(210)}) {};
\node [circle,fill=black,inner sep=1pt,label=right:{{\tiny $e$}}] (C) at ({cos(330)},{sin(330)}) {};
\node [circle,fill=black,inner sep=1pt,label=above:{{\tiny $1$}}] (D) at ({2*cos(90)},{2*sin(90)}) {};
\node [circle,fill=black,inner sep=1pt,label=below:{{\tiny $e^3$}}] (E) at ({2*cos(210)},{2*sin(210)}) {};
\node [circle,fill=black,inner sep=1pt,label=below:{{\tiny $e^5$}}] (F) at ({2*cos(330)},{2*sin(330)}) {};

\draw (B)--(F)--(A)--(D)--(E)--(F)--(D)--(C)--(A)--(B)--(C)--(E)--(F)--(C);
\draw (A)--(E);
\end{tikzpicture}
\end{tabular}\\
Figure 2
\end{center}
\begin{theorem}
Let $G$ be a group. Then $\p^*(G)$ is $1$-planar if and only if $\omega(G)\subseteq\{1,2,3,4,5,6,7\}$.
\end{theorem}
\begin{proof}
The same as in the proof of Theorem \ref{properplanar}, we can show that $\omega(G)\subseteq\{1,2,3,4,5,6,7\}$. The converse is obvious since every element of order $7$ of $G$ along with its nontrivial powers gives a complete connected component of $\p^*(G)$ isomorphic to $K_6$ and the remaining elements of $G$, by Theorem \ref{properplanar}, induce a planar graph.
\end{proof}

An \textit{almost-planar} graph $\Gamma$ is a graph with an edge $e$ whose removal is a planar graph.
\begin{theorem}
Let $G$ be a group. Then $\p(G)$ is almost-planar if and only if $w(G)\subseteq\{1,2,3,4\}$ or $G$ is isomorphic to one of the groups $\Z_5$, $\Z_6$, $D_{10}$, $D_{12}$, $\Z_3\rtimes\Z_4$ or $\Z_5\rtimes_F\Z_4$.
\end{theorem}
\begin{proof}
Clearly, $w(G)\subseteq\{1,2,3,4,5,6\}$. If $w(G)\subseteq\{1,2,3,4\}$, then $\p(G)$ is planar and we are done. Thus we may assume that $5\in w(G)$ or $6\in w(G)$. First suppose that $5\in w(G)$. If $G$ has two distinct cycles $\gen{x}$ and $\gen{y}$ or order $5$. Then the subgraph induced by $\gen{x}\cup\gen{y}$ is isomorphic to $K_5\cdot K_5$, which is not almost-planar. Hence $G$ has a unique cyclic subgroup $\gen{x}$ of order $5$. Then $\gen{x}\normal G$ and $G/C_G(x)$ is a cyclic group of order dividing $4$. However, $C_G(x)=\gen{x}$ from which it follows that $|G|$ divides $20$ and hence $G\cong\Z_5$, $D_{10}$, $\Z_5\rtimes_F\Z_4$. Now, suppose that $5\notin w(G)$ but $6\in w(G)$. If $G$ has two distinct cyclic subgroups $\gen{x}$ and $\gen{y}$ of order $6$, then a simple verification shows that $\gen{x}\cup\gen{y}$ is never almost-planar in either of cases $\gen{x}\cap\gen{y}$ has one, two or three elements. Thus $G$ has a unique cyclic subgroup $\gen{x}$ of order $6$. Clearly, $\gen{x}\normal G$ and $C_G(x)=\gen{x}$, which implies that $|G|$ divides $12$. Therefore, $G\cong\Z_6$, $D_{12}$ or $\Z_3\rtimes\Z_4$. The converse is straightforward.
\end{proof}
\begin{theorem}
Let $G$ be a group. Then $\p^*(G)$ is almost planar if and only if $\omega(G)\subseteq\{1,2,3,4,5,6\}$.
\end{theorem}
\begin{proof}
If $\p^*(G)$ is almost planar, then since $K_6$ is not almost planar, by Theorem \ref{properplanar}, $\omega(G)\subseteq\{1,2,3,4,5,6\}$. The converse is obvious for by Theorem \ref{properplanar}, $\p^*(G)$ is planar.
\end{proof}

A simple graph is called \textit{maximal planar} if it is planar but the graph obtained by adding any new edge is not planar.
\begin{theorem} 
Let $G$ be a group. Then $\p(G)$ is maximal planar if and only if $G$ is a cyclic group of order at most four.
\end{theorem}
\begin{proof}
Suppose that $\p(G)$ is maximal planar. Since $\p(G)$ is planar, by Theorem \ref{planar}, $\omega(G)\subseteq\{1,2,3,4\}$. If $G$ has two different maximal cycles $\gen{x}$ and $\gen{y}$ of order $2$, $3$ or $4$, then the addition of the edge $\{x,y\}$ results in a planar graph, which is a contradiction. Therefore, the maximal cycles in $G$ give rise to a partition for $G$. Since $\p(G)$ is connected, it follows that $G$ is cyclic, from which the result follows. The converse is clear.
\end{proof}
\begin{theorem} 
Let $G$ be a group. Then $\p^*(G)$ is maximal planar if and only if $G$ is a cyclic group of order at most five.
\end{theorem}
\begin{proof}
Suppose that $\p^*(G)$ is maximal planar. Since $\p^*(G)$ is planar, by Theorem \ref{properplanar}, $\omega(G)\subseteq\{1,2,3,4,5,6\}$. If $G$ has two different cycles $\gen{x}$ and $\gen{y}$ of order $4$ or $6$ whose intersection is nontrivial, then by Theorem \ref{properplanar}, the addition of the edge $\{x,y\}$ results in a planar graph, which is a contradiction. Therefore, the maximal cycles in $G$ give rise to a partition of $G$. Since $\p^*(G)$ is connected, it follows that $G$ is cyclic, from which the result follows. The converse is clear.
\end{proof}

It is worth noting that the structure of groups with elements of orders at most six is known and we refer the interested reader to \cite{rb-wjs,ndg-vdm,tlh-wjs,fl-blw,dvl,vdm,bhn,wjs-cy,wjs-wzy,cy-sgw-xjz,akz-vdm} for details.
\section {Toroidal (proper) power graphs}
Let $S_k$ be the sphere with $k$ handles (or connected sum of $k$ tori), where $k$ is a non-negative integer, that is, $S_k$ is an oriented surface of genus $k$. The genus of a graph $\Gamma$, denoted by $\gamma(\Gamma)$, is the minimal integer $k$ such that the graph can be embedded in $S_k$ such that the edges intersect only in the endpoints. A graph with genus $0$ is clearly a planar graph. A graph with genus $1$ is called a \textit{toroidal graph}. We note that if $\Gamma'$ is a subgraph of a graph $\Gamma$, then $\gamma(\Gamma')\leq\gamma(\Gamma)$. For complete graph $K_n$ and complete bipartite graph $K_{m,n}$, it is well known that
\[\gamma(K_n)=\left\lceil\frac{(n-3)(n-4)}{12}\right\rceil\]
if $n\geq 3$ and 
\[\gamma(K_{m,n})=\left\lceil\frac{(m-2)(n-2)}{4}\right\rceil\]
if $m,n\geq 2$. (See \cite{gr-2} and \cite{gr-1}, respectively). Thus
\begin{itemize}
\item $\gamma(K_n)=0$ for $n=1,2,3,4$,
\item $\gamma(K_n)=1$ for $n=5,6,7$,
\item $\gamma(K_n)\geq 2$ for $n\geq 8$,
\item $\gamma(K_{m,n})=0$ for $m=0,1$ or $n=0,1$, 
\item $\gamma(K_{m,n})=1$ for $\{m,n\}=\{3\}$, $\{3,4\}$, $\{3,5\}$, $\{3,6\}$, $\{4\}$,
\item $\gamma(K_{m,n})\geq2$ for $\{m,n\}=\{4,5\}$, or $m,n\geq3$ and $m+n\geq10$.
\end{itemize}
Given a connected graph $\Gamma$, we say that a vertex $v$ of $\Gamma$ is a cut-vertex if $\Gamma-v$ is disconnected. A block is a maximal connected subgraph of $\Gamma$ having no cut-vertices. The following result of Battle, Harary, Kodama, and Youngs gives a powerful tool for computing genus of various graphs.
\begin{theorem}[Battle, Harary, Kodama, and Youngs, \cite{jb-fh-yk-jwty}]\label{blocks}
Let $\Gamma$ be a graph and $\Gamma_1,\ldots,\Gamma_n$ be blocks of $\Gamma$. Then 
\[\gamma(\Gamma)=\gamma(\Gamma_1)+\cdots+\gamma(\Gamma_n).\]
\end{theorem}
\begin{theorem}\label{toroidal}
Let $G$ be a group. Then $\p(G)$ is a toroidal graph if and only if $G\cong\Z_5$, $\Z_6$, $\Z_7$, $D_{10}$, $D_{12}$, $D_{14}$, $\Z_3\rtimes\Z_4$, $\Z_5\rtimes_F\Z_4$ or $\Z_7\rtimes_F\Z_3$.
\end{theorem}
\begin{proof}
First assume that $\p(G)$ is a toroidal graph. The same as in the proof of Theorem \ref{properplanar}, we can show that  $\omega(G)\subseteq\{1,2,3,4,5,6,7\}$. If $\omega(G)\subseteq\{1,2,3,4\}$, then by Theorem \ref{planar}, $\p(G)$ is planar, which is a contradiction. Thus $\omega(G)\cap\{5,6,7\}\neq\emptyset$. Since $\gen{x}$ is a block of $\p(G)$ when $|x|=5,7$, and $\gen{x}$ is a subgraph of a block of $\p(G)$ when $|x|=6$, Theorem \ref{blocks} shows that $\omega(G)\subseteq\{1,2,3,4,5\}$, $\{1,2,3,4,6\}$ or $\{1,2,3,4,7\}$. Moreover, $G$ has at most one subgroup of order $5$ and $7$. 

If $7\in\omega(G)$, then $G$ has a unique cyclic subgroup $\gen{x}$ of order $7$. Clearly, $\gen{x}\normal G$ and $C_G(x)=\gen{x}$. Since $G/C_G(x)$ is isomorphic to a subgroup of $\Aut(\gen{x})$ and $6\notin\omega(G)$, $G/C_G(x)$ is a cyclic group of order at most $3$, from which it follows that $G\cong\Z_7$, $D_{14}$ or $\Z_7\rtimes_F\Z_3$. Similarly, if $5\in\omega(G)$, then we can show that $G\cong\Z_5$, $D_{10}$ or $\Z_5\rtimes_F\Z_4$. 

Finally, suppose that $6\in\omega(G)$. If $G$ has a unique cyclic subgroup of order $6$, say $\gen{x}$, then $\gen{x}\normal G$ and a simple verification shows that $C_G(x)=\gen{x}$. Since $G/\gen{x}$ is isomorphic to a subgroup of $\Aut(\gen{x})$, it follows that $G\cong\Z_6$, $D_{12}$ or $\Z_3\rtimes\Z_4$. Now suppose that $G$ has at least to distinct cyclic subgroups of order $6$. If $G$ has two distinct cyclic subgroups $\gen{x}$ and $\gen{y}$ of order $6$ such that $\gen{x}\cap\gen{y}=\gen{a}\cong\Z_2$, then the subgraph induced $\gen{x}\cup\gen{y}\setminus\{a\}$ is isomorphic to $K_5\cdot K_5$ and by Theorem \ref{blocks}, $\p(G)$ is not toroidal, which is a contradiction. Since the cycles of order $6$ sharing an element of order $3$ are blocks, by Theorem \ref{blocks}, either $G$ all cyclic subgroups of order $6$ have the same subgroup of order $3$ in common. Clearly, $G$ has at most three cyclic subgroups of order $6$ for otherwise $G$ has four cyclic subgroups $\gen{x}$, $\gen{y}$, $\gen{z}$ and $\gen{w}$ of order $6$ and the subgraph induced by $\gen{x}\cup\gen{y}\cup\gen{z}\cup\gen{w}$ has a subgraph isomorphic to $K_{3,8}$, which is a contradiction. If $G$ has three distinct cyclic subgroups $\gen{x}$, $\gen{y}$ and $\gen{z}$ of order $6$, then by using \cite{en-wm}, it follows that the subgraph induced by $\gen{x}\cup\gen{y}\cup\gen{z}$ has genus $2$, which is a contradiction. Hence, $G$ has exactly two distinct subgroups of order $6$, say $\gen{x}$ and $\gen{y}$. Let $H=N_G(\gen{x})$. Then $[G:H]\leq2$. Since $G$ has two cyclic subgroups of order $6$, a simple verification shows that $C_H(x)=\gen{x}$. On the other hand, $H/\gen{x}$ is isomorphic to a subgroup of $\Aut(\gen{x})\cong\Z_2$. Hence $|H|$ divides $12$ and consequently $|G|$ divides $24$, which results in a contradiction for there are no such groups. The converse is straightforward.
\end{proof} 
\begin{theorem}\label{propertoroidal}
Let $G$ be a group. Then $\p^*(G)$ is a toroidal graph if and only if $G\cong \Z_7$, $\Z_8$, $D_{14}$, $D_{16}$, $Q_{16}$, $QD_{16}$ or $\Z_7\rtimes\Z_3$.
\end{theorem}
\begin{proof}
Suppose $\p^*(G)$ is toroidal. The same as in the proof of Theorem \ref{properplanar}, we can show that $\omega(G)\subseteq\{1,2,3,4,5,6,7,8\}$. If $\omega(G)\subseteq\{1,2,3,4,5,6\}$, then by Theorem \ref{properplanar}, $\p^*(G)$ is planar, which is a contradiction. Thus $\omega(G)\cap\{7,8\}\neq\emptyset$. On the other hand, by Theorem \ref{blocks}, $\{7,8\}\not\subseteq\omega(G)$. Therefore, either $\omega(G)\subseteq\{1,2,3,4,5,6,7\}$ or $\omega(G)\subseteq\{1,2,3,4,5,6,8\}$. If $7\in\omega(G)$, then again by Theorem \ref{blocks}, $G$ has a unique cyclic subgroup $\gen{x}$ of order $7$. Clearly, $\gen{x}\normal G$ and $C_G(x)=\gen{x}$. Since $G/\gen{x}$ is isomorphic to a subgroup of $\Aut(\gen{x})\cong\Z_6$, it follows that $G\cong\Z_7$, $D_{14}$, $\Z_7\rtimes\Z_3$ or $\Z_7\rtimes\Z_6$. Now, suppose that $8\in\omega(G)$. If $G$ has two distinct cyclic subgroups $\gen{x}$ and $\gen{y}$ of order $8$, then the subgraph induced by $\gen{x}\cup\gen{y}$ has a subgraph isomorphic to $K_5\cdot K_5$ or $2K_5$, which is a contradiction by Theorem \ref{blocks}. Therefore, $G$ has a unique cyclic subgroup $\gen{x}$ of order $8$. Then $\gen{x}\normal G$ and $C_G(x)=\gen{x}$. Since $G/\gen{x}$ is isomorphic to a subgroup of $\Aut(\gen{x})\cong\Z_2\times\Z_2$, by using GAP \cite{tgg}, we obtain $G\cong\Z_8$, $D_{16}$, $Q_{16}$ or $QD_{16}$. The converse is obvious.
\end{proof} 
\section{Projective (proper) power graphs}
The \textit{real projective plane} is a non-orientable surface, which can be represented on plane by a circle with diametrically opposed points identified.

Let $N_k$ be the connected sum of $k$ projective planes, where $k$ is a non-negative integer. The corsscap number of a graph $\Gamma$, denoted by $\overline{\gamma}(\Gamma)$, is the minimal integer $k$ such that $\Gamma$ can be embedded in $N_k$ such that the edges intersect only in the endpoints. A graph with crosscap $0$ is clearly a planar graph. A graph with crosscap $1$ is called a \textit{projective graph}. Clearly, if $\Gamma'$ is a subgraph of a graph $\Gamma$, then $\overline{\gamma}(\Gamma')\leq\overline{\gamma}(\Gamma)$. For complete graph $K_n$ and complete bipartite graph $K_{m,n}$, it is well known that
\[\overline{\gamma}(K_n)=\begin{cases}\left\lceil\frac{(n-3)(n-4)}{6}\right\rceil,&n\geq3\emph{ and }n\neq7,\\3,&n=7,\end{cases}\]
and 
\[\overline{\gamma}(K_{m,n})=\left\lceil\frac{(m-2)(n-2)}{2}\right\rceil\]
if $m,n\geq 2$. (See \cite{gr-2} and \cite{gr-1}, respectively). Thus
\begin{itemize}
\item $\overline{\gamma}(K_n)=0$ for $n=1,2,3,4$,
\item $\overline{\gamma}(K_n)=1$ for $n=5,6$,
\item $\overline{\gamma}(K_n)\geq 2$ for $n\geq 7$,
\item $\overline{\gamma}(K_{m,n})=0$ for $m=0,1$ or $n=0,1$, 
\item $\overline{\gamma}(K_{m,n})=1$ for $\{m,n\}=\{3\}$, $\{3,4\}$,
\item $\overline{\gamma}(K_{m,n})\geq2$ for $m,n\geq3$ and $m+n\geq8$.
\end{itemize}

A graph $\Gamma$ is \textit{irreducible} for a surface $S$ if $\Gamma$ does not embed in $S$ but any proper subgraph of $\Gamma$ embeds in $S$. Kuratowski's theorem states any graph which is irreducible for the plane is homomorphic to either $K_5$ or $K_{3,3}$. Glover, Huneke and Wang \cite{hg-ph-cw} constructed a list of $103$ pairwise non-homomorphic graphs which are irreducible for the real projective plane. Also, Archdeacon in \cite{da} proved this list is complete in the sense that a graph can be embedded on the real projective plane if and only if it has no subgraph homomorphic to any of the $103$ given graphs. For example, the graphs $K_5\cdot K_5$, $2K_5$, $K_{3,3}\cdot K_{3,3}$, $2K_{3,3}$, $K_{3,3}\cdot K_5$ and $K_{3,3}\cup K_5$ are irreducible for the real projective plane. Furthermore, the irreducible graphs in \cite{hg-ph-cw} show that the graph $K_7$ is not projective.
\begin{theorem}\label{projective}
Let $G$ be a finite group. Then $\p(G)$ is projective if and only if $G\cong\Z_5$, $\Z_6$, $D_{10}$, $D_{12}$, $\Z_3\rtimes\Z_4$ or $\Z_5\rtimes_F\Z_4$.
\end{theorem}
\begin{proof}
Suppose $\p(G)$ is projective. The same as in the proof of Theorem \ref{properplanar}, it can be easily seen that $\omega(G)\subseteq\{1,2,3,4,5,6\}$. If $\omega(G)\subseteq\{1,2,3,4\}$, then by Theorem \ref{planar}, $\p(G)$ is planar, which is a contradiction. If $5,6\in\omega(G)$, then the subgraph induced by $\gen{x}\cup\gen{y}$, where $\gen{x}$ and $\gen{y}$ are distinct subgroups of orders $5$ and $5$, or $5$ and $6$, respectively, contains a subgraph isomorphic to $K_5\cdot K_5$, which is a contradiction. Thus $\omega(G)\subseteq\{1,2,3,4,5\}$ or $\{1,2,3,4,6\}$.  If $5\in\omega(G)$, then $G$ has a unique normal cyclic subgroup $\gen{x}$ of order $5$. Clearly, $C_G(x)=\gen{x}$. Hence $G/C_G(x)$ is isomorphic to a subgroup of $\Aut(\gen{x})\cong\Z_4$, which implies that $|G|$ divides $20$. Therefore $G\cong\Z_5$, $D_{10}$ or $\Z_5\rtimes_F\Z_4$. Finally, suppose that $6\in\omega(G)$. If $\gen{x}$ and $\gen{y}$ are two distinct subgroups of order $6$, then $\gen{x}\cap\gen{y}\cong\Z_3$ for otherwise the subgraph induced by $\gen{x}\cup\gen{y}$ has a subgraph isomorphic to $K_5\cdot K_5$, which is a contradiction. Since $G$ has no subgraphs isomorphic to $K_{3,6}$ it follows that $G$ has at most two cyclic subgroups of order $6$. If $G$ has two distinct cyclic subgroups $\gen{x}$ and $\gen{y}$ of order $6$, then since $N_G(\gen{x})/C_G(x)$ is isomorphic to a subgroup of $\Aut(\gen{x})\cong\Z_2$ and $C_G(x)=\gen{x}$, it follows that $|N_G(\gen{x})|$ divides $12$. On the other hand, $\gen{x}$ has at most two conjugates, namely $\gen{x}$ and $\gen{y}$, which implies that $[G:N_G(\gen{x})]\leq2$. Thus $|G|$ divides $24$. A simple verification by GAP \cite{tgg} shows that there are no groups of order dividing $24$ which admit exactly two cyclic subgroups of order $6$, which contradicts our assumption. Therefore, $G$ has a unique cyclic subgroup of order $6$, say $\gen{x}$. Then $\gen{x}\normal G$ and $C_G(x)=\gen{x}$. Hence $G/\gen{x}$ is isomorphic to a subgroup of $\Aut(\gen{x})\cong\Z_2$, which implies that $G\cong\Z_6$, $D_{12}$ or $\Z_3\rtimes\Z_4$. The converse is straightforward.
\end{proof}
\begin{theorem}\label{properprojective}
Let $G$ be a finite group. Then $\p^*(G)$ is projective if and only if $G\cong\Z_7$, $D_{14}$, $\Z_7\rtimes\Z_3$ or $\Z_7\rtimes\Z_6$.
\end{theorem}
\begin{proof}
Suppose $\p^*(G)$ is projective. The same as in the proof of Theorem \ref{properplanar}, it can be easily seen that $\omega(G)\subseteq\{1,2,3,4,5,6,7\}$. If $\omega(G)\subseteq\{1,2,3,4,5,6\}$, then by Theorem \ref{properplanar}, $\p(G)$ is planar, which is a contradiction. Thus $7\in\omega(G)$. If $G$ has two distinct cyclic subgroups of order $7$, then $\p^*(G)$ has a subgraph isomorphic to $2K_6$, which is impossible. Thus $G$ has a unique cyclic subgroup $\gen{x}$ of order $7$. Then $\gen{x}\normal G$ and $G/C_G(x)$ is isomorphic to a subgroup of $\Aut(\gen{x})\cong\Z_6$. However, $C_G(x)=\gen{x}$, which implies that $|G|$ divides $42$. Therefore $G\cong\Z_7$, $D_{14}$, $\Z_7\rtimes\Z_3$ or $\Z_7\rtimes\Z_6$. The converse is obvious.
\end{proof}
\begin{acknowledgment}
The authors would like to thank prof. Myrvold for providing us with the toroidal graph testing  algorithm used in the proof of Theorem \ref{toroidal}.
\end{acknowledgment}

\end{document}